\def\marker{\>\hbox{${\vcenter{\vbox{
    \hrule height 0.4pt\hbox{\vrule width 0.4pt height 6pt
    \kern6pt\vrule width 0.4pt}\hrule height 0.4pt}}}$}\>}

\documentclass[12pt]{article}

\usepackage[left=2cm,right=2cm,top=3cm,bottom=3cm]{geometry}

\usepackage{amsmath, amssymb, latexsym, amsthm}

\usepackage{graphicx}
\usepackage{fullpage}

\usepackage{tikz}
\usetikzlibrary{shapes}

\newtheorem{theorem}{Theorem} 
\newtheorem{theorem*}{Theorem} 

\theoremstyle{definition}

\theoremstyle{remark}

{\end{enumerate}}

\usepackage[hang,flushmargin]{footmisc}

\title{Improved algorithm to determine 3-colorability of graphs with the minimum degree at least 7}
\author{Nicholas Crawford\footnotemark[1], Sogol Jahanbekam\footnotemark[2], and Katerina Potika\footnotemark[3]}

\date{}

\begin{document}

\maketitle

\begin{abstract}
Let $G$ be an $n$-vertex graph with the maximum degree $\Delta$ and the minimum degree $\delta$. We  give  algorithms with complexity $O(1.3158^{n-0.7~\Delta(G)})$ and $O(1.32^{n-0.73~\Delta(G)})$ that determines if $G$ is 3-colorable, when $\delta(G)\geq 8$ and $\delta(G)\geq 7$, respectively.

\end{abstract}

\renewcommand{\thefootnote}{\fnsymbol{footnote}}
\footnotetext[1]{
Department of Mathematics and Statistics, San Jose State University, San Jose, CA; {\tt nicholas.crawford@sjsu.edu.}\\
}
\footnotetext[2]{
Department of Mathematics and Statistics, San Jose State University, San Jose, CA; {\tt sogol.jahanbekam@sjsu.edu.}\\{Research was supported in part by by NSF grant CMMI-1727743 and Woodward Fund.}
}

\footnotetext[3]{
 Computer Science Department, San Jose State University, San Jose, CA;
{\tt katerina.potika@sjsu.edu }.}
\renewcommand{\thefootnote}{\arabic{footnote}}

\baselineskip18pt

\section{Introduction}

A coloring of the vertices of a graph is \textit{proper} if adjacent vertices receive different colors. A graph $G$ is  $k$-\textit{colorable} if it has a proper coloring using $k$ colors.  The \textit{chromatic number} of a graph $G$, written as $\chi(G)$, is the smallest integer $k$ such that $G$ is $k$-colorable.

The proper coloring problem is one of the  most studied
problems in graph theory. To determine the chromatic number of a graph, one should find the smallest integer $k$ for which the graph is $k$-colorable.  The $k$-colorability problem, for $k\geq 3$, is one of the classical NP-complete
problems \cite{W}.

Even approximating the chromatic number has been shown to be a very hard problem.  Lund and Yannakakis \cite{LY}  have shown that there is an $\epsilon$ such that the chromatic number of a general $n$-vertex graph cannot be approximated with ratio $n^{\epsilon}$ unless
$P = NP$.



In 1971,  Christofides obtained the first non-trivial algorithm computing the chromatic
number of $n$-vertex graphs  running in  $n!n^{O(1)}$ time \cite{CH}.  Five years later Lawler \cite{LA} used dynamic programming and enumerations of maximal independent sets to improve it to an algorithm with running time $O^*(2.4423^n)$.  Later  the running time was improved by Eppstein \cite{EP}.  The best-known complexity for determining the chromatic number of graphs is due to Bj\"orklund, Husfeldt, and Koivisto \cite{BHK} who used a combination
of inclusion-exclusion and dynamic programming to  develop a $O(2^n)$ algorithm to determine the chromatic number of $n$-vertex graphs.


 The $k$-colorability problem for small values of $k$, like 3 and 4 is also a highly-studied problem that  has  attracted a lot of attention. Not only this problem has its own importance, but also improving the bounds for small values of $k$ could be used to improve the bound for higher values of $k$ and as a result, improve the complexity of the general coloring problem. The fastest known algorithm deciding if a graph is 3-colorable or not runs in  $O(1.3289^n)$ time and is due to Beigel and Eppstein \cite{BE}. The fastest known algorithm for 4-colorability  runs in $O(1.7272^n)$ and is due to Fomin, Gaspers, and Saurabh \cite{FGS}.
 
 In this paper, we prove the following.

 \begin{theorem}\label{main}
Let $G$ be an $n$-vertex graph with maximum degree $\Delta$ and minimum degree $\delta$, where $\delta(G)\geq 8$.  We can determine in $O(1.3158^{n-0.7\Delta})$ time if $G$ is $3$-colorable or not.
\end{theorem}

\begin{theorem}\label{main2}
Let $G$ be an $n$-vertex graph with maximum degree $\Delta$ and minimum degree $\delta$, where $\delta(G)\geq 7$.  We can determine in  $O(1.32^{n-0.73\Delta})$ time if $G$ is $3$-colorable or not.
\end{theorem}


For smaller minimum degree conditions, results similar to the
statements of Theorems \ref{main} and \ref{main2}
can be proved, but the complexity would increase. For example, the
3-colorability of a graph with minimum degree 6 can be
determined in $O(1.368^{n-.7d(v}))$ time.
This result is not an improvement compared to that of Beigel and
Eppstein \cite{BE} however, because  $1.368>1.3289$.

\section{Definitions, Notation, and Tools }

In this section we define the terms and notation we use to prove Theorems \ref{main} and \ref{main2}.

For a graph $G$ with vertex set $V(G)$ and edge set $E(G)$, we denote the minimum degree by $\delta(G)$ and the maximum degree by $\Delta(G)$. We suppose all graphs studied in this note are simple. Let $v$ be a vertex in $G$. The degree of $v$ in $G$ is denoted by $d_G(v)$ or simply $d(v)$ (when there is no fear of confusion).  The open neighborhood of $v$ in $G$, denoted by $N_G(v)$ (or simply $N(v)$), is the set of neighbors of $v$ in $G$ and $N^2(v)$ denotes the set of vertices in $G$ that are in distance (exactly) 2 from $v$. Therefore $N(v)\cap N^2(v)=\emptyset$. The closed neighborhood of $v$ in $G$, denoted by $N[v]$, is equal to $N(v)\cup \{v\}$. 

Let $A$ be a subset of $V(G)$. The graph $G[A]$ is the induced subgraph of $G$ with vertex set $A$.  Let $u$ and $v$ be two vertices of $G$. The graph $G/uv$ is the graph obtained from $G$ after contracting (identifying) the vertices $u$ and $v$ in $G$ and replacing multiple edges by one edge, so that the resulting graph is simple.

Suppose for each vertex $v$ in $V(G)$, there exists a list of colors denoted by $L(v)$. A \textit{proper list coloring} of $G$ is a choice function that maps every vertex $v$ to a color in the list $L(v)$ in such a way that the coloring is proper. A graph is $k$-\textit{choosable} if it has a proper list coloring whenever each vertex has a list of size $k$. 


A \textit{Boolean expression} is a logical statement that is either TRUE or FALSE. In computer science, the \textit{Boolean satisfiability problem} (abbreviated to SAT) is the problem of determining if there exists an interpretation that satisfies a given Boolean expression.  The 3-satisfiability problem or \textit{3-SAT} problem is a special case of SAT problem,      where the Boolean expression  can be divided into clauses such that every clause contains  three literals.

The constraint satisfiability  problem is a satisfiability problem which is not necessarily Boolean. In an $(r,t)-CSP$  instance, we are given a collection of $n$ variables, each of which can be given one of up to $r$ different colors and a set of constraints, where each constraint is expressed using $t$ variables, i.e. certain color combinations are forbidden for $t$ variables.


By the above definition 3-SAT is the same as $(2,3)$-CSP.  It was proved in \cite{BE}  that each $(a,b)$-CSP instance is equivalent to a $(b,a)$-CSP instance. Therefore any 3-SAT is equivalent to a $(3,2)$-CSP instance.



The following result was proved by Beigen and Eppstein    in \cite{BE}. We will apply this theorem in the proof of Theorem \ref{main}.

\begin{theorem}\label{BE2} \cite{BE}
$n$-variable (3,2)-CSP instances can be solved in  $O(1.3645^n)$ time.
\end{theorem}

\section{Proof of Theorem \ref{main}}

To prove Theorem \ref{main} we prove the following stronger theorem.

\begin{theorem}\label{4}
Let $G$ be a graph and $v$ be a vertex  in  $G$  with the property that all vertices in $V(G)-(N[v]\cup N^2(v))$ have degree at least 8 in $G$, then we can determine in time $O(1.3158^{n-0.7d(v)})$  if $G$ is $3$-colorable or not.
\end{theorem}

\begin{proof}
We apply induction on $n-d(v)$  to prove the assertion. Since $G$ is simple, we have $d(v)\leq n-1$. Therefore $n-d(v)\geq 1$.

 When $n-d(v)=1$, the graph $G$ has a vertex $v$ of degree $n-1$. In this case $G$ is 3-colorable if and only if $G-v$ is 2-colorable. Since 2-colorability can be determined in polynomial time (for example using a simple Breadth First Search algorithm we can determine in linear time if the graph is bipartite), the assertion holds in this case.


Let us assume that for any $n$-vertex graph $H$, with a vertex $v$ of degree $d(v)$, where $n-d(v)\leq k$ and $k\geq 1$, we can determine if $H$ is 3-colarable in  $O(1.3158^{n-0.7d(v)})$ time, given all vertices in $V(H)-(N[v]\cup N^2(v))$ have degree at least $8$ in $H$.

We prove that the Theorem holds when the graph $G$ is an $n$-vertex graph having a vertex $v$ with  $n-d(v)=k+1$, where all vertices in $V(G)-(N[v]\cup N^2[v])$ have degree at least $8$ in $G$.

 If there are three vertices $u_1,u_2,u_3$ in $N(v)$ with $u_1u_2,u_2u_3\in E(G)$ (see Figure 1), then $u_1u_3\in E(G)$ implies that $G$ is not 3-colorable, and $u_1u_3\not\in E(G)$ implies that the vertices $u_1$ and $u_3$ must get the same colors in any proper 3-coloring of $G$. As a result, we can identify $u_1$ and $u_3$ in $G$ and study the smaller graph. Hence we may suppose that $G[N(v)]$ has no vertex of degree at least 2.

\begin{figure}[htp]
    \centering
    \includegraphics[width=4cm]{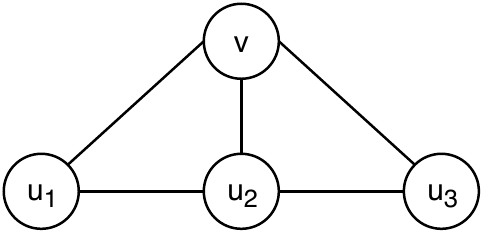}
    \caption{When $G[N(v)]$ has a vertex $u_2$ of degree at least 2.}
    \label{fig:path2}
\end{figure}

We consider three cases.

\subsection{Case 1: When  $d(v)>0 .309n$.}
\label{proof:Case1}

In this case we  transfer the problem into a (3,2)-CSP problem with $n-d(v)-1$ vertices. With no loss of generality we may suppose that in any coloring the color of $v$ is 1.  As a result, the vertices in $N(v)$ must get colors in $\{2,3\}$. We create a (3,2)-CSP on $V(G)-N[v]$ in such a way that $G$ is 3-colorable if and only if the (3,2)-CSP problem has a solution.

Suppose $N(v)=\{u_1,\ldots,u_r,w_1,\ldots,w_r,z_1,\ldots,z_t\}$, where $u_1w_1,\ldots,u_rw_r$ are the only edges with both ends in $N(v)$. This holds because  $G[N(v)]$ has no vertex of degree at least 2.

\begin{figure}[htp]
    \centering
    \includegraphics[width=4cm]{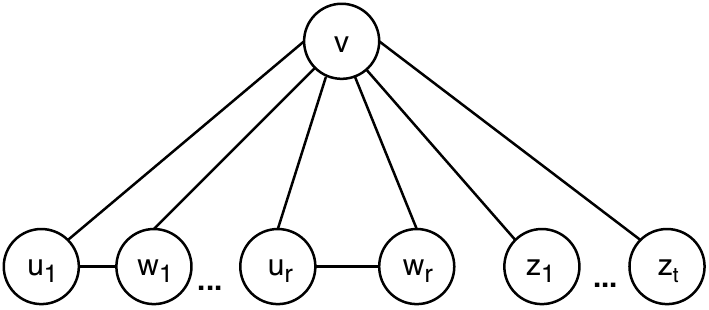}
    \caption{Notation of Case 1.}
    \label{fig:path2}
\end{figure}

If $u_i$ and $w_i$,  for some integer $i$, have a common neighbor $y$ in $N^2(v)$, then in any proper 3-coloring of $G$ the vertices $v$ and $y$ must get the same color.  As a result we can contract $v$ and $y$ in $G$ and study the smaller graph. Hence we may suppose that $u_i$ and $w_i$ have no common neighbors in $N^2(v)$.

Let $H$ be a graph with $V(H)=V(G)-N[v]$. We define a (3,2)-CSP on $H$ as follows. 

For vertices $x,y\in V(H)$, if $xy\in E(G)$, then we need to avoid patterns 1-1, 2-2, and 3-3 on $x$ and $y$, i.e. we need $(x,y)\neq (1,1), (2,2),(3,3)$. If $x$ and $y$ have a common neighbor in $N(v)$ (in $G$), then we need to avoid patterns 2-3 and 3-2 on $x$ and $y$ (i.e. $(x,y)\neq (2,3),(3,2)$), since otherwise we cannot extend the coloring on $V(H)$ to a proper 3-coloring of $G$. Finally, if $xu_i,yw_i\in E(G)$, then we need to avoid patterns 2-2 and 3-3 on $x$ and $y$ (i.e. $(x,y)\neq (2,2),(3,3)$), since otherwise we cannot extend the coloring on $V(H)$ to a proper 3-coloring of $G$. 

By the above construction of the (3,2)-CSP  on $H$, the graph $G$ is 3-colorable if and only if the (3,2)-CSP on $H$ has a solution. Note that constructing $H$ takes a polynomial time process and by Theorem \ref{BE2} determining if the (3,2)-CSP instance  on $H$ has a solution or not has complexity $O((1.3645)^{n-d(v)-1})$. Since
$O((1.3645)^{n-d(v)})\subseteq O(1.3157^{n-0.7d(v)})$ for $d(v)>0.309n$. Therefore a polynomial factor of $O(1.3157^{n-0.7d(v)})$ is a subset of $O(1.3158^{n-0.7d(v)})$, as desired.

\subsection{Case 2. When $V(G)=N[v]\cup N^2(v)$ and $d(v)\leq 0.309n$.}
\label{proof:Case2}
In this case with no loss of generality we may suppose that in any coloring the  color of $v$ is 1. As a result, the vertices in $N(v)$ must get colors in $\{2,3\}$. Therefore there are at most $2^{d(v)}$ different possibilities for the colors of the vertices in $N[v]$. Since $V(G)=N[v]\cup N^2(v)$, all vertices in $V(G)-N[v]$ have at least one neighbor in $N[v]$. 

Let $c$ be a proper coloring over $G[N[v]]$ using colors $2$ and $3$. As a result, to extend this coloring to a proper coloring of $G$ each vertex in $N^2(v)$ must avoid at least one color (the color(s) of its neighbor(s) in $N[v]$). Hence each vertex in $N^2(v)$ has a list of size at most 2, such that $c$ can be extended to a proper coloring of $G$ if and only if there exists a proper list coloring on $N^2(v)$. Note that we can determine in polynomial time if there exists a proper list coloring on the vertices of a graph, when each list has size at most 2 (see \cite{KT} ).

Since there are at most $2^{d(v)}$ proper coloring on $N(v)$ in which all vertices get colors in $\{2,3\}$,  we can determine in a polynomial factor of $2^{d(v)}$ if $G$ is 3-colorable or not. Since $d(v)\leq 0.309n$, we have $2^{d(v)}\leq  (1.31578)^{n-0.7d(v)}$. Hence $2^{d(v)}\subseteq  O(1.31578)^{n-0.7d(v)}$, which implies $poly(n) 2^{d(v)}\subseteq  O(1.3158)^{n-0.7d(v)}$, as desired.

\subsection{Case 3. When $V(G)\neq N[v]\cup N^2(v)$ and $d(v)\leq 0.309n$.}
\label{proof:Case3}
  
  Let $x$ be a vertex in $V(G)-(N[v]\cup N^2(v))$. In any proper 3-coloring of $G$, if it exists, the vertex $x$ either gets the same color as $v$ or $x$ receives a different color than $v$.  Therefore it is enough to determine if any of the graphs $G/xv$ and $G\cup xv$ are 3-colorable.  Recall that by our hypothesis $d(x)\geq 8$.

Let $H=G/xv$ and $H'=G\cup xv$.  The graph $H$ has $n-1$ vertices. Since $x$ has degree at least 8 in $G$ and since it has no common neighbor with $v$, we have $d_H(v)\geq d_G(v)+8$. Similarly, we have $n(H')=n(G)$ and $d_{H'}(v)=d_G(v)+1$. Therefore by the induction hypothesis, we can determine in $O(1.3158^{n-1-0.7(d(v)+8)})$ time if the graph $H$ is 3-colorable and we can determine in $O(1.3158^{n-0.7(d(v)+1)})$ time if the graph $H'$ is 3-colorable. Therefore to determine if $G$ is 3-colorable, we require an algorithm of complexity at most $O(1.3158^{n-0.7d(v)-6.6})+O(1.3158^{n-0.7d(v)-0.7})$. 

Note that $1.3158^{n-0.7d(v)-6.6}+1.3158^{n-0.7d(v)-0.7}< 1.3158^{n-0.7d(v)}$. Therefore the assertion holds.
\end{proof}

\section{Proof of Theorem \ref{main2}}

The proof of Theorem \ref{main2} is very similar to the proof of Theorem \ref{main}. To avoid redundancy we skip the parts of the proof that are similar. We prove the following stronger result.

\begin{theorem}
Let $G$ be a graph and $v$ be a vertex  in  $G$  with the property that all vertices in $V(G)-(N[v]\cup N^2(v))$ have degree at least 7 in $G$, then we can determine in  $O(1.32^{n-0.73d(v)})$ time if $G$ is $3$-colorable or not.
\end{theorem}

\begin{proof}
We apply induction on $n-d(v)$. When $n-d(v)=1$, the graph $G$ has a vertex $v$ of degree $n-1$. In this case $G$ is 3-colorable if and only if $G-v$ is 2-colorable (can be determined in polynomial time), the assertion holds in this case.


Assume that for any $n$-vertex graph $H$, with a vertex $v$ of degree $d(v)$, where $n-d(v)\leq k$ and $k\geq 1$, we can determine if $H$ is 3-colarable in  $O(1.32^{n-0.73d(v)})$ time, given all vertices in $V(H)-(N[v]\cup N^2(v))$ have degree at least $7$ in $H$.

We prove that the statement holds when an $n$-vertex graph $G$ has a vertex $v$ with  $n-d(v)=k+1$, where all vertices in $V(G)-(N[v]\cup N^2[v])$ have degree at least $7$ in $G$.

Similar to the argument in the proof of Theorem \ref{4} there are no three vertices $u_1,u_2,u_3$ in $N(v)$ with $u_1u_2,u_2u_3\in E(G)$ (see Figure 1).

We consider the following three cases.

\noindent
\textit{Case 1.} When  $d(v)> 0.309n$.

\noindent
\textit{Case 2.} When $V(G)=N[v]\cup N^2(v)$ and $d(v)\leq 0.309n$.

\noindent
  \textit{Case 3.} When $V(G)\neq N[v]\cup N^2(v)$ and $d(v)\leq 0.309n$.

  The proof of Cases 1 and 2 is almost identical to that in the proof of Theorem \ref{4} with the small difference that the base of the complexity    ($1.3158$) must be replaced by $1.32$ and $1.3157$ and $1.31578$ in Cases 1 and 2 must be replaced by $1.3199$. Hence we move forward to the proof of Case 3, which is also similar to that in the proof of Theorem \ref{4}.

  Let $x$ be a vertex in $V(G)-(N[v]\cup N^2(v))$. Note that $G$ is 3-colorable if and only if  $G/xv$ or $G\cup xv$ is 3-colorable.  Therefore it is enough to determine if any of the graphs $G/xv$ and $G\cup xv$ is 3-colorable.  Recall that by our hypothesis $d(x)\geq 7$.

Let $H=G/xv$ and $H'=G\cup xv$.  The graph $H$ has $n-1$ vertices and $d_H(v)\geq d_G(v)+7$. Similarly, we have $n(H')=n(G)$ and $d_{H'}(v)=d_G(v)+1$.  Hence, by the hypothesis, we can determine in $O(1.32^{n-1-0.73(d(v)+7)})$ time if the graph $H$ is 3-colorable, and we can determine in $O(1.32^{n-0.73(d(v)+1)})$ time if the graph $H'$ is 3-colorable. All together, to determine if $G$ is 3-colorable, the algorithm has a complexity of at most $O(1.32^{n-0.73d(v)-6.11})+O(1.32^{n-0.73d(v)-0.73})$. 

Since $1.32^{n-0.73d(v)-6.11}+1.32^{n-0.73d(v)-0.73}< 1.32^{n-0.73d(v)}$,  the assertion holds.
\end{proof}

\noindent{\bf Acknowledgment:}  The authors would like to thank the anonymous referees, whose suggestions greatly improved the exposition of this paper.

\end{document}